\documentclass{amsart}
\usepackage{amsmath,amssymb,graphicx}
\usepackage[usenames]{color}

\newtheorem{thm}{Theorem}
\newtheorem{lem}[thm]{Lemma}

\newtheorem{prop}[thm]{Proposition}

\theoremstyle{definition}
\newtheorem{defn}[thm]{Definition}

\newtheorem{rmk}[thm]{Remark}
\newtheorem{exmp}[thm]{Example}

\newcommand{\CPb}{\overline{\mathbb{CP}}{}^{2}}
\newcommand{\CP}{{\mathbb{CP}}{}^{2}}

\title[Lantern substitution and new symplectic 4-manifolds with ${b_{2}}^{+} = 3$]
{Lantern substitution and new\\ 
symplectic 4-manifolds with ${b_{2}}^{+} = 3$} 

\begin{document}

\author{Anar Akhmedov}
\address{School of Mathematics \\
University of Minnesota \\
Minneapolis, MN, 55455, USA}
\email{akhmedov@math.umn.edu}

\author{Jun-Yong Park}
\address{School of Mathematics \\ 
University of Minnesota \\
Minneapolis, MN, 55455, USA}
\email{junepark@math.umn.edu}

\date{November 11, 2011.  Revised on September 7, 2012}



\begin{abstract} Motivated by the construction of H. Endo and Y. Gurtas, changing a positive relator in Dehn twist generators of the mapping class group by using lantern substitutions, we show that $4$-manifold $K3\#2\CPb$ equipped with the genus two Lefschetz fibration can be rationally blown down along six disjoint copies of the configuration $C_2$. We compute the Seiberg-Witten invariants of the resulting symplectic $4$-manifold, and show that it is symplectically minimal. Using our example, we also construct an infinite family of pairwise non-diffeomorphic irreducible symplectic and non-symplectic $4$-manifolds homeomorphic to $M = 3\CP\# (19-k)\CPb$ for $1 \leq k \leq 4$. 

\end{abstract}

\maketitle

\section{Introduction} 

\par There is a beautiful interplay between the algebra and the topology when one studies a Lefschetz fibration structure on a smooth $4$-manifold. To every Lefschetz fibration over $\mathbb{S}^2$, one can associate a $\emph{factorization}$ of the identity element as a product of positive Dehn twists in the mapping class group of the fiber, and conversely, for such a factorization in the mapping class group, there is a corresponding Lefschetz fibration over $\mathbb{S}^2$ \cite{GS}. By the remarkable work of Donaldson \cite{D1, D2}, every closed symplectic $4$-manifold admits a structure of Lefschetz pencil, which can be blown up at its base points to yield a Lefschetz fibration. Conversely, Gompf \cite{GS} showed that if $g \geq 2$, then the total space of a genus $g$ Lefschetz fibration admits a symplectic structure. As one changes the identity word by conjugations and relations of the mapping class group, the corresponding Lefschetz fibration also changes topologically. There are many efforts in trying to understand what all the relations in the mapping class groups mean topologically. One of the well understood relations is the lantern relation, which corresponds to the symplectic operation of rational blowdown \cite{FS1, EG}.

\par  In this article, we start with the genus two Lefschetz fibration on $K3\#2\CPb$ over $\mathbb{S}^2$ with global monodromy given by the relation $\varrho = (t_{c_5}t_{c_4}t_{c_3}t_{c_2}t_{c_1})^6=1$ in the mapping class group  $M_{2}$ of a closed genus two surface, where each $c_i$ is a simple closed curve as in Figure ~ \ref{fig:Sigma2} and $t_{c_i}$ is a right handed Dehn twist along the curve $c_i$, $i = 1, \cdots, 5$. We factorize the monodromy of the given Lefschetz fibration by a series of conjugations and braid relations to get a word upon which we can perform six lantern relation substitutions. Applying these lantern substitutions change the total space of our Lefschetz fibration topologically as a six rational blowdowns on $K3\#2\CPb$. Furthemore, using the Seiberg-Witten invariants, we show that the resulting symplectic $4$-manifold is homeomorphic but not diffeomorphic to $3\CP\#15\CPb$ and symplectically minimal. We would like to point out that the goal of this paper is not to construct exotic smooth structures on very small $4$-manifolds with ${b_{2}}^{+} = 3$, but rather to use the lantern relation substitutions to study smooth structures on various $4$-manifolds. The study of exotic $4$-manifolds with small Euler characteristics has been carried out by the first author and B. D. Park in \cite{A, AP1, AP2}.

\par In the following three sections, we present some background material and recall some results which will be needed in this paper. Section 2 discusses some well-known relations in the mapping class group, which will be used in our computations and study. In Section 3, we give a brief background information on Lefschetz fibrations, provide some examples that will be used to illustrate discussions in the paper, and prove various results that will be needed in the sequel. In Sections 4 and 5, we recall the rational blowdown technique of Fintushel and Stern, provide the theorem proved by H. Endo and Y. Gurtas relating the lantern substitution to the rational blowdown operation, state results of R. Gompf on rational blowdown along smooth $-4$ sphere, and discuss the knot surgery operation of Fintushel-Stern, respectively. Finally, in section 6, we prove our main theorems. 

\section{Mapping Class Groups} 
\par Let $\Sigma_{g}$ denote a $2$-dimensional, closed, oriented, and connected surface of genus $g>0$ surface. 

\begin{defn}

Let $Diff^{+}\left( \Sigma_{g}\right)$ denote the group of all orientation-preserving diffeomorphisms $\Sigma_{g}\rightarrow \Sigma_{g},$ and $ Diff_{0}^{+}\left(
\Sigma_{g}\right) $ be the subgroup of $Diff^{+}\left(\Sigma_{g}\right) $ consisting of all orientation-preserving diffeomorphisms $\Sigma_{g}\rightarrow \Sigma_{g}$ that are isotopic to the identity. \emph{The mapping class group} $M_{g}$ of $\Sigma_{g}$ is defined to be the group of isotopy classes of orientation-preserving diffeomorphisms of $\Sigma_{g}$, i.e.,
\[
M_{g}=Diff^{+}\left( \Sigma_{g}\right) /Diff_{0}^{+}\left(
\Sigma_{g}\right) .
\]
\end{defn}

\begin{defn}
\par Let $\alpha$ be a simple closed curve on $\Sigma_{g}$. A \emph{right handed Dehn twist} $t_\alpha$ about $\alpha$ is the isotopy class of a self-diffeomophism of $\Sigma_{g}$ obtained by cutting the surface $\Sigma_{g}$ along $\alpha$ and gluing the ends back after rotating one of the ends $2\pi$ to the right. 
\end{defn}

Notice that the conjugate of a Dehn twist is again a Dehn twist. Indeed, if $f: \Sigma_{g}\rightarrow \Sigma_{g}$ is an orientation-preserving diffeomorphism, then it is easy to check that $f \circ t_\alpha \circ f^{-1} = t_{f(\alpha)}$. Next, we briefly mention some relations that hold in the mapping class group $M_{g}$.
This elementary fact and relations will be used quite often in our computation in Section 6.

\subsection{Commutativity and Braid Relation} Let $\alpha$ and $\beta$ be two simple closed curves on $\Sigma_{g}$.

\begin{lem} \label{com&braid.lem}
If $\alpha$ and $\beta$ are disjoint, then we have the following commutativity relation: $t_{\alpha }t_{\beta}=t_{\beta }t_{\alpha }.$
If $\alpha$ and $\beta$ transversely intersect at a single point, then the corresponding Dehn twists satisfy the following braid relation: $t_{\alpha }t_{\beta }t_{\alpha }=t_{\beta }t_{\alpha }t_{\beta }.$
\end{lem}

For a proof see \cite{I}.

\subsection{Lantern Relation} Let $\Sigma_{0,4}$ be a sphere with 4 boundary components.

\begin{lem} \label{lantern.lem}
If $\delta_1, \delta_2, \delta_3, \delta_4$ are the boundary curves of $\Sigma_{0,4}$ and $\alpha$, $\beta$, $\gamma$ are 
the simple closed curves as shown in Figure~\ref{fig:lan},
then we have \[t_{\gamma }t_{\beta }t_{\alpha
}=t_{\delta_{1}}t_{\delta_{2}}t_{\delta_{3}}t_{\delta_{4}},
\] where $t_{\delta_{i}},$ $1\leq i\leq 4,$ denote the Dehn twists about $\delta_{i}.$
\end{lem}

For a proof see \cite{I}.

\begin{figure}[ht]
\begin{center}
\includegraphics[scale=.43]{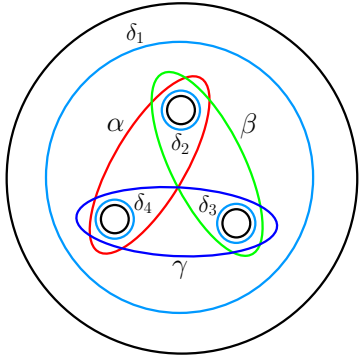}
\caption{Lantern Relation}
\label{fig:lan}
\end{center}
\end{figure}

This relation was known to Dehn. Later on it was rediscovered by D. Johnson and named as \emph{lantern relation} by him \cite{I}. For more results on lantern relation see \cite{Mar}.

\section{Lefschetz Fibration}

Let us first recall the definition of Lefschetz fibration. 

\begin{defn}\label{defn:lef.fibration} Let $X$ be a compact, connected, oriented, smooth $4$-manifold. A \emph{Lefschetz fibration} on $X$ is a smooth map $f : X \longrightarrow \Sigma_{h}$,
where $\Sigma_{h}$ is a compact, oriented, smooth $2$-manifold of genus $h$, such that $f$ is surjective and each critical point of $f$ has an orientation preserving chart on which $f: \mathbb{C}^2 \longrightarrow \mathbb{C}$ is given by $f(z_1, z_2) = {z_1}^2 + {z_2}^2$.
\end{defn}

It is a corollary of the Sard's theorem that $f$ is a smooth fiber bundle away from finitely many critical points, say $p_1, \cdots, p_k$. The genus of the regular fiber of $f$ is defined to be the genus of the Lefschetz fibration.  If a fiber passes through critical point set $p_{1}, \cdots, p_{k}$, then it is called a singular fiber which is an immersed surface with a single transverse self-intersection. A singular fiber can be described by its monodromy, an element of the mapping class group $M_{g}$, where $g$ is the genus of the Lefschetz fibration. This element is a right handed Dehn twist along a simple closed curve on $\Sigma$, called the \emph {vanishing cycle}. If this curve is a nonseparating curve, then the singular fiber is called \emph{nonseparating}, otherwise it is called \emph{separating}. For a genus $g$ Lefschetz fibration over $S^2$, the product of right handed Dehn twists $t_{\alpha_{i}}$ along the vanishing cycles $\alpha_i$, for $i = 1, \cdots , k$, gives us the global monodromy of the Lefschetz fibration, the relation $t_{\alpha_1} \cdot t_{\alpha_2} \cdots  t_{\alpha_k} = 1$ in $M_{g}$. Conversely, such a relation defines a genus $g$ Lefschetz fibration over $S^2$ with the vanishing cycles $\alpha_1, \cdots, \alpha_k$.

Let $c_{1}$, $c_{2}$, $c_{3}$, $c_{4}$, and $c_{5}$ be the simple closed curves as in Figure~\ref{fig:Sigma2}. For convenience we shall denote the right handed Dehn twists $t_{c_i}$ along the curve $c_i$ by $c_{i}$. It is well-known that the following relations hold in the mapping class group $M_2$:  

\begin{equation}
\begin{array}{l}
(c_1c_2c_3c_4{c_5}^2c_4c_3c_2c_1)^2 = 1,  \\
(c_1c_2c_3c_4c_5)^6 = 1,  \\ 
(c_1c_2c_3c_4)^{10} = 1. 
\end{array}
\end{equation}

For each relation above, we have a corresponding genus two Lefschetz fibrations over $S^2$ with total spaces $\CP\#13\CPb$, $K3\#2\CPb$, and the Horikawa surface $H$, respectively. In this paper, we will consider genus two Lefschetz fibration on $K3\#2\CPb$ with global monodromy $(c_1c_2c_3c_4c_5)^6 = 1$. 

The following result is well-known. For the convenience of the reader, we sketch the proof.

\begin{lem}\label{E} The genus two Lefschetz fbration on $K3\#2\CPb$ over $S^2$ with the monodromy factorization $(c_1c_2c_3c_4c_5)^6 = 1$ can be obtained as the double branched covering of $\CP\#\CPb$ branched along a smooth algebraic curve $B$ in the linear system $|6\tilde{L}|$, where $\tilde{L}$ is the proper transform of line $L$ in $\CP$ avoiding the blown-up point.
\end{lem}

\begin{proof} We closely follow the proof of Lemma 3.1 in \cite{Ar1}, see also the discussion in \cite{AK}. Let $D$ denote a degree $d$ algebraic curve in $\CP$. Fix a generic projection $\pi : \CP \setminus {pt} \rightarrow \mathbb{CP}^1$ whose pole does not belong to $D$. According to the work of B. Moishezon and M. Teicher, the braid monodromy describing a degree $d$ branch curve $D$ in $\CP$ is given by a braid factorization. In fact, it is shown that the braid monodromy around the point at infinity in $\mathbb{CP}^1$, which is given by the central element $\Delta^2$ in $B_{d}$, can be written as the product of the monodromies about the critical points of the projection map $\pi$. More precisely, the following factorization $\Delta^2 = (\sigma_{1} \cdots \sigma_{d-1})^{d}$ holds in the braid group $B_{d}$, where $\sigma_{i}$ is a positive half-twist exchanging two points, and fixing the remaining $d-2$ points.

We first degenerate the smooth algebraic curve $B$ in $\CP\#\CPb$ into a union of $6$ lines in a general position (see Figure~\ref{fig:branch}). The braid group factorization corresponding to the configuration $B$ is given by $\Delta^2 = (\sigma_{1} \sigma_{2} \sigma_{3} \sigma_{4} \sigma_{5})^{6}$. By lifting this braid factorization to the mapping class group of the genus two surface, we obtain that the monodromy factorization $(c_1c_2c_3c_4c_5)^6 = 1$ for the corresponding double branched covering.

Notice that a regular fiber of the given fibration is a two fold cover of a sphere with homology class $f = h - e_{1}$  branched over $6$ points, where $h$ denotes the hyperplane class in $\CP$. Thus a regular fiber is a surface of genus two. The exceptional sphere $e_{1}$ in $\CP\#\CPb$, which intersects $f = h - e_{1}$ positively at one point, gives rise to two disjoint $-1$ sphere sections to the given genus two Lefschetz fibration on $K3\#2\CPb$.

\end{proof}

\begin{figure}[ht]
\begin{center}
\includegraphics[scale=.3]{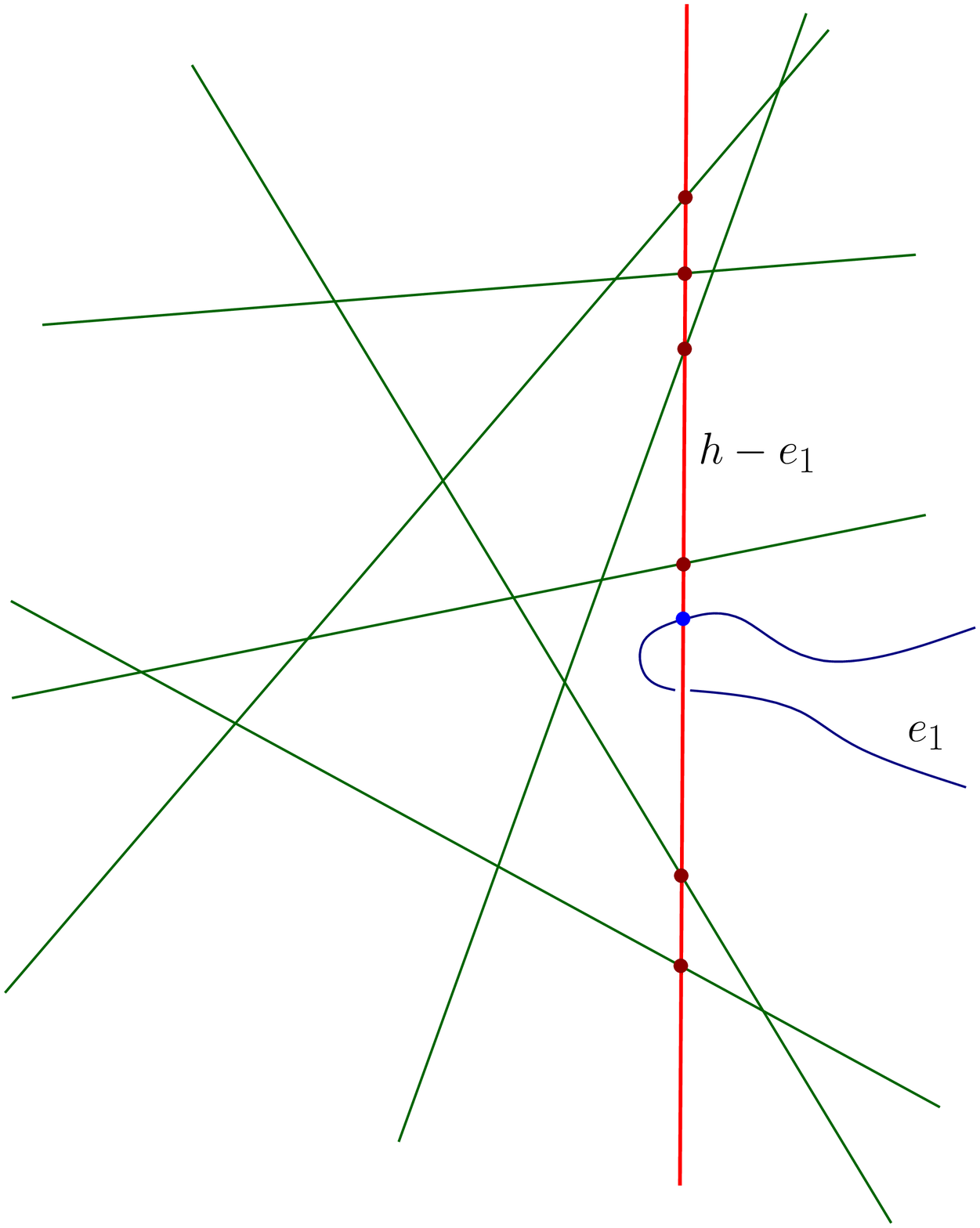}
\caption{Branch Locus for $K3\#2\CPb$}
\label{fig:branch}
\end{center}
\end{figure}

Below we prove key Proposition, which plays an important role in the proof of our main Theorems. The following proposition also gives an alternative and convenient way of thinking the genus two Lefschetz fibration on $K3\#2\CPb$ over $S^2$ with the monodromy $(c_1c_2c_3c_4c_5)^6 = 1$.

\begin{prop}\label{E1} The genus two Lefschetz pencil on $K3$ with two base points given above can be constructed by symplectic fiber summing two copies of $E(1) = \CP\#9\CPb$ along a regular torus fiber. 
\end{prop}

\begin{proof} Since the fiber of the elliptic fibration on $E(1) = \CP\#9\CPb$ is a blow up of a generic cubic curve in $\CP$, its homology class is equal to $f = 3h - e_{1}  -  e_{2} - \cdots  -  e_{9}$, where $h$ denotes the hyperplane class in  $\CP$ and $e_{i}$ is the homology class of the exceptional sphere of the $i^{th}$ blow-up. We first consider the pencil of lines in  $\CP$ all passing through the fixed point $p_{1}$ away from the cubic curve $C$ in $\CP$. Observe that each generic line $L$, which is a sphere of self-intersection $1$, in this pencil intersects the cubic curve $C$ at three distinct points by Bezout's Theorem. Furthermore, by applying the Riemann-Hurwitz formula to the restriction of the map $\pi : \CP \setminus {p_{1}} \rightarrow \mathbb{CP}^1$ to $C$, we compute the degree of the ramification divisor $R$: $deg(R) = 2g(C) - 2 - 3(2g(\mathbb{CP}^1) -2) = 6$. Consequently, for a generic smooth cubic $C$, there are exactly six lines in the pencil above that are tangent to $C$. 

Next, we choose the regular torus fiber $F$ along which the fiber sum of two copies of $E(1) = \CP\#9\CPb$ will be performed. Since the generic elliptic fiber of $E(1) = \CP\#9\CPb$ intersects the lines of the corresponding pencil at three points (see Figure~\ref{fig:genustwopencil}), the generic line of the pencil in each $E(1)$ intersects the boundary of $E(1) \setminus \nu(F)$ in three disjoint circles. We choose a gluing diffeomorphism $\psi = id_{F} \times (complex \ conjugation)$, that identities these circles as in Figure~\ref{fig:braided} to obtain a pencil of genus two curves in  $E(2) = K3$ surface. Since the pencils in each copy of $E(1)$ are holomorphic and the gluing map on the boundary $3$-torus is identity on the elliptic fiber, the resulting genus two pencil is holomorphic as well. By holomorphically blowing up this genus two pencils at base points of the pencil $p_1$ and $p_2$ in $K3$ surface, we obtain the genus two holomorphic fibration on $K3\#2\CPb$. This holomorphic fibration has six singular fibers resulting from the gluing of six tangent lines mentioned above in each copy of $E(1)$. These singular fibers are not a Lefschetz type, and each singularity is topologically a sphere. Furthermore, using the analysis of the singular fibers, we see that each singular fiber can be perturbed into five Lefschetz type singularities with non-separating vanishing cycles (see the braid monodromy discussion in \cite{Ar1}, page 5). Finally, using the result of B. Siebert and G. Tian, Theorem A in \cite{ST}, on holomorphicity of genus two Lefschetz fibration, we see that the fibration is isomorphic to the one given in Lemma \ref{E}. The isomorphism of fibrations also follows from the proof of Lemma \ref{E} (see also discussion in \cite{Ar1}, page 5) by considering the degree, and the braid monodromy of the ramification divisor. First, we view the fiber sum of two copies of $E(1)$ as a two-fold ramified branched cover of $E(1)$ along the smooth divisor in the class $6h - 2e_{1} - 2e_{2} - \cdots -2e_{9} = 2F_{E(1)} = 2(-K_{E(1)})$, twice the anticanonical divisor of $E(1)$. A pencil of lines in $\CP$ with one base point gives rise to a pencil of lines in $E(1)$ with one base point assuming that we blow up the pencil $9$ times away from the basepoint. Since a generic line in the pencil, which has class $h$, intersects the ramification divisor $6h - 2e_{1} - 2e_{2} - \cdots -2e_{9}$ at six points, it determines a genus two pencil under two-fold cover. Notice that the braid monodromy of a smooth plane curve of degree six (i.e., of a ramification divisor), can be computed by degenerating it into the union of six lines in generic position and given by $(c_1c_2c_3c_4c_5)^6 = 1$. 

\end{proof}

\begin{rmk}

This description of the genus two pencil allows us to see rim tori and Gompf nuclei in elliptic surface $K3$ (See Example~\ref{Ex} below and \cite{gompf1} for an explanation). 

\end{rmk}

\begin{figure}[ht]
\begin{center}
\includegraphics[scale=.48]{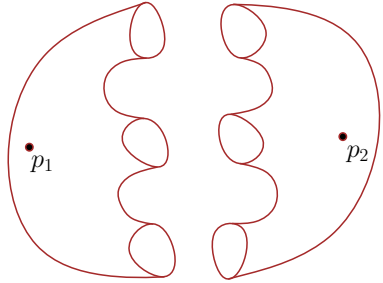}
\caption{A Genus Two Surface via Fiber Sums of $E(1)$}
\label{fig:braided}
\end{center}
\end{figure}

\begin{figure}[ht]
\begin{center}
\includegraphics[scale=.53]{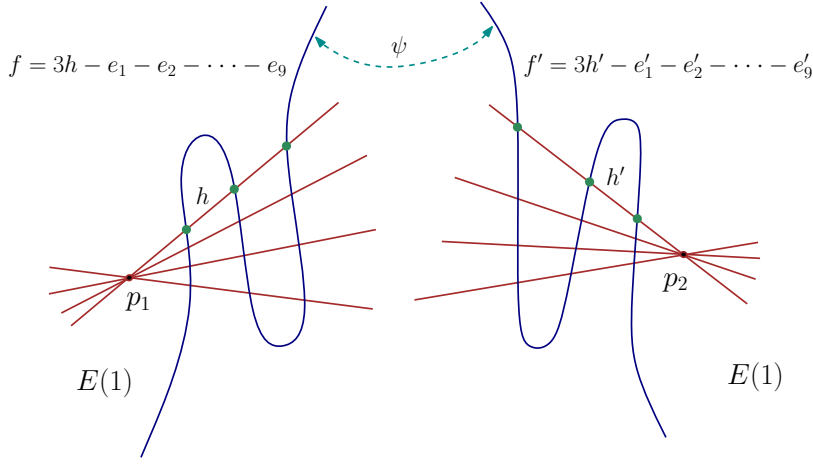}
\caption{The Pencil of Genus Two Curves in $K3$ surface}
\label{fig:genustwopencil}
\end{center}
\end{figure}

\begin{exmp}\label{Ex} In this example, we study $K3$ surfaces $E(2) =  E(1) \#_{\mathbb{T}^2} E(1)$ in some detail. As a consequence of our discussion, we will also derive some useful facts about the genus two Lefschetz pencil on $K3$ with two base points given above. Let us think of $K3$ surface as the fiber sum of two copies of $E(1) = \CP\# 9\CPb$ along a torus fiber as in Proposition~\ref{E1}. We choose the following basis for the intesection form of $E(1)$:  $ <f = 3h - e_1 - ... - e_9, \ e_9, \ e_1 - e_2, \ e_2 - e_3, \ \cdots, \ e_7 - e_8, \ -h + e_6 + e_7 + e_8 >$. Note that the last $8$ classes can be represented by spheres of self-intersection $-2$ and generate the intersection matrix  $-E_8$, where $E_8$ the matrix corresponding to the Dynkin diagram of the exceptional Lie algebra $E_8$. The class $f$ is fiber of an elliptic fibration on $E(1) = \CP\# 9\CPb$ and $e_9$ is a sphere section of self-intersection $-1$. When we perform the fiber sum along torus to obtain $E(2)$, it is not hard to see the surfaces that generate the intersection form $2(-E_8) \oplus 3H$ for $E(2)$, where $H$ is a hyperbolic pair. The two copies of the Milnor fiber $\Phi(1) \in E(1)$ are in $E(2)$, providing $16$ spheres of self-intersection $-2$ (corresponding to the classes $\{ e_1 - e_2, \ e_2 - e_3, \ \cdots, \ e_7 - e_8, \ -h + e_6 + e_7 + e_8 \}$ and $\{ {e_1}' - {e_2}', \ {e_2}' - {e_3}', \ \cdots, \ {e_7}' - {e_8}', \ -{h}' + {e_6}' + {e_7}' + {e_8}' \}$ mentioned above), realize two copies of $-E_8$. One copy of hyperbolic pair $H$ comes from an identification of the torus fibers $f$ and $f'$, and a sphere section $\sigma$ of self-intersection $-2$ obtained by sewing the sphere sections $e_{9}$ and ${e_{9}}'$, i.e. from the Gompf's nucleus $N(2)$ in $E(2)$. The remaining two copies of $H$ come from $2$ rim tori and their dual $-2$ spheres (see discussion in \cite{GS}, page 73)). These $22$ classes ($19$ spheres and $3$ tori) generate $H_2$ of $E(2)$. Since $c_1(E(n)) = (2-n)f$, $E(2)$ has a trivial canonical class. As we can see from the Figure~\ref{fig:genustwopencil}, the class of the genus two surface of square $2$ of the genus two Lefschetz pencil on $K3$ is $h + h'$, we simply add the homology classes of the surfaces. As a consequence, the class of the genus two fiber in $K3\#2\CPb$ is given by $h + h' - E_{1} - E_{2}$, where $E_{1}$ and $E_{2}$ are the homology classes of the exceptional spheres of the blow-ups at the points $p_1$ and $p_2$. We can also verify the symplectic surface $\Sigma$, given by the class $h + h' - E_{1} - E_{2}$, has genus two by applying the adjunction formula to $(K3\#2\CPb, \Sigma)$: $g(\Sigma) = 1 + 1/2(K_{K3\#2\CPb} \cdot [\Sigma] + [\Sigma]^2) = 1 + ((E_{1} + E_{2}) \cdot (h + h' - E_{1} - E_{2}) + (h + h' - E_{1} - E_{2})^2)/2 = 1 + (2 + 0)/2 = 2$. The reader can notice from the intersection form of $K3$ that both rim tori has no intersections with the genus surface in the pencil given by the homology class $h + h'$. Thus, the genus two fiber $\Sigma$ is disjoint from the rim tori that descend to $K3\#2\CPb$.    

\end{exmp}

\section{Rational Blowdown and Lantern Relations}

The basic idea of the rational blowdown surgery is that if a smooth $4$-manifold $X$ contains a particular configuration $C_p$ of transversally intersecting $2$-spheres whose boundary is the lens space $L(p^2, 1 - p)$, then one can replace $C_p$ with rational ball $B_p$ to construct a new manifold $X_p$. If one knows the Seiberg-Witten invariants of the original manifold $X$, then one can determine the same invariants of $X_p$. Below we briefly discuss the rational blow-down. We refer the reader to \cite{FS1,P1} for full details.

Let $p \geq  2$ and $C_p$ be the smooth $4$-manifold obtained by plumbing disk bundles over the $2$-sphere according to the following linear diagram  

 \begin{picture}(100,60)(-90,-25)
 \put(-12,3){\makebox(200,20)[bl]{$-(p+2)$ \hspace{6pt}
                                  $-2$ \hspace{96pt} $-2$}}
 \put(4,-25){\makebox(200,20)[tl]{$u_{p-1}$ \hspace{25pt}
                                  $u_{p-2}$ \hspace{86pt} $u_{1}$}}
  \multiput(10,0)(40,0){2}{\line(1,0){40}}
  \multiput(10,0)(40,0){2}{\circle*{3}}
  \multiput(100,0)(5,0){4}{\makebox(0,0){$\cdots$}}
  \put(125,0){\line(1,0){40}}
  \put(165,0){\circle*{3}}
\end{picture}

\noindent where each vertex $u_{i}$ of the linear diagram represents a disk bundle over $2$-sphere with the given Euler number. 

According to Casson and Harer \cite{CH}, the boundary of $C_p$ is the lens space $L(p^2, 1 - p)$ which also bounds a rational ball $B_p$ with $\pi_1(B_p) = \mathbb{Z}_p$ and $\pi_1(\partial B_p) \rightarrow  \pi_1(B_p)$ surjective. If $C_p$ is embedded in a $4$-manifold $X$ then the rational blowdown manifold $X_p$ is obtained by replacing $C_p$ with $B_p$, i.e., $X_p = (X \setminus C_p) \cup B_p$. If $X$ and $X \setminus C_p$ are simply connected, then so is $X_p$. 

\begin{lem}\label{thm:rb} $b_{2}^{+}(X_p) = {b_2}^{+}(X)$, $\sigma(X_p) = \sigma(X) + (p-1)$, ${c_1}^{2}(X_p) = {c_1}^2(X) + (p-1)$, and $\chi_{h}(X_p) = \chi_{h}(X)$.
\end{lem}

\begin{proof} 

Notice that the manifold $C_{p}$ is negative definite, we have ${b_{2}}^{+}(X_{p}) = {b_{2}}^{+}(X)$ and ${b_{2}}^{-}(X_{p}) = {b_{2}}^{-}(X) - (p-1)$. Thus,  $\sigma(X_p) = \sigma(X) + (p-1)$. Using the formulas ${c_1}^{2} = 3\sigma +2e$ and $\chi_{h} = (\sigma +e)/4$, we have ${c_{1}}^{2}(X_{p}) = 3\sigma(X_{p}) + 2e(X_{p}) = 3(\sigma(X)+(p-1)) + 2(e(X)-(p-1)) = {c_{1}}^{2}(X) + (p-1)$ and $\chi_{h}(X_p) = (\sigma(X)+(p-1) + e(X)-(p-1))/4 = \chi_{h}(X)$.  

\end{proof}

\begin{thm}\label{SW1} \cite{FS1, P1}. Suppose $X$ is a smooth 4-manifold with $b_{2}^{+}(X) > 1$ which contains a configuration $C_{p}$. If $L$ is a SW basic class of $X$ satisfying $L\cdot u_{i} = 0$ for any i with $1 \leq i \leq p-2$  and $L\cdot u_{p-1} = \pm p$, then $L$ induces a SW basic class $\bar L$ of $X_{p}$ such that $SW_{X_{p}}(\bar L) = SW_{X}(L)$.  

\end{thm}

\begin{thm}\label{SW2} \cite{FS1, P1} If a simply connected smooth $4$-manifold $X$ contains a configuration $C_{p}$, then the SW-invariants of $X_{p}$ are completely determined by those of $X$. That is, for any characteristic line bundle $\bar{L}$ on $X_{p}$ with $SW_{X_{p}}(\bar{L}) \ne 0$, there exists a characteristic line bundle $L$ on $X$ 
such that $SW_{X}(L) = SW_{X_{p}}(\bar{L})$.

\end{thm}

In this paper we only use the rational blowdown surgery along configuration $C_{2}$, i.e. the rational blowdowns along the $-4$ spheres. 

In Section 6, we shall use the following theorem of H. Endo and Y. Gurtas.

\begin{thm}\label{bd} Let $\varrho,\varrho'$ be positive relators of $\mathcal{M}_g$ 
and $M_{\varrho},M_{\varrho'}$ the corresponding Lefschetz fibrations over $S^2$, respectively.
If $\varrho'$ is obtained by applying 
a lantern substitution to $\varrho$, 
then the $4$-manifold $M_{\varrho'}$ is a rational blowdown of $M_{\varrho}$ 
along a configuration $C_2\subset M_{\varrho}$. 
\end{thm}

\medskip

We will also need the following lemmas, which are due to R. Gompf, to analyze the symplectic $4$-manifolds constructed in Section 6. For the proof we refer the reader to \cite{gompf1} and \cite{Dorf2}. See also the work of J. Dorfmeister \cite{Dorf1, Dorf2} who gives a related criteria on symplectic minimality and how the symplectic Kodaira dimension changes under the rational blowdown along a symplectic $-4$ sphere (see also related work in \cite{DZ}).  

\begin{lem}\label{n=1}
Let $(X,V_X)$ be a relatively minimal smooth pair with $V_X$ an embedded $-4$ sphere.  If $X$ contains a smoothly embedded exceptional sphere transversely intersecting the hypersurface $V_X$ in a single positive point, then the manifold obtained under $-4$ blow-down of $V_X$ is diffeomorphic to the blow-down of $X$ along this sphere.
\end{lem}

\begin{lem}\label{n=2}
Let $(X,V_X)$ be a relatively minimal smooth pair with $V_X$ an embedded $-4$ sphere. If $X$ contain two disjoint smoothly embedded exceptional spheres each transversely intersecting the hypersurface $V_X$ in a single positive point, then the manifold obtained under $-4$ blow-down of $V_X$ is diffeomorphic to the blow-down of $X$ along one of these spheres.  
  
\end{lem}

\section{Knot Surgery} Let $X$ be a $4$-manifold with ${b_2}^{+}(X) > 1$ and contain a homologically essential torus $T$ of self-intersection $0$. Let $N(K)$ be a tubular neighborhood of $K$ in $S^3$, and let $T \times D^2$ be a tubular neighborhood of $T$ in $X$. The knot surgery manifold $X_{K}$ is defined by $X_K = (X \setminus (T \times D^2)) \cup (S^1 \times (S^3 \setminus N(K))$ where two pieces are glued in a way that the homology class of $[pt \times \partial  D^2]$ is identifed with $[pt \times \lambda]$ where $\lambda$ is the class of the longitude of knot $K$. Fintushel and Stern proved the theorem that shows Seiberg-Witten invariants of $X_{K}$ can be completely determined by the Seiberg-Witten invariant of $X$ and the Alexander polynomial of $K$ \cite{FS2}. Furthermore, if $X$ and $X \setminus T$ are simply connected, then so is $X_K$. 

\begin{thm}\label{thm:knotsurgery} Suppose that $\pi_{1}(X) = \pi_{1}(X \setminus T) = 1$ and $T$ lies in a cusp neighborhood in $X$. Then $X_{K}$ is homeomorphic to $X$ and Seiberg-Witten invariants of $X_{K}$ is $SW_{X_K} = SW_{X} \cdot \Delta_{K}(t^2)$, where $t = t_{T}$ and $\Delta_{K}$ is the symmetrized Alexander polynomial of $K$. If the Alexander polynomial $\Delta_{K}(t)$ of knot $K$ is not monic then $X_K$ admits no symplectic structure. Moreover, if $X$ is symplectic and $K$ is a fibered knot, then $X_{K}$ admits a symplectic structure.
\end{thm}

We refer the reader to \cite{FS2} for the details.

\section{Construction of exotic 4-manifolds via lantern substitution}

In this section, we first construct a simply connected, minimal symplectic $4$-manifold $X$ homeomorphic but not diffeomorphic to\/ $3\CP\#15\CPb$ starting from $K3\#2\CPb$ and applying the sequence of six rational blowdowns via lantern substitutions. Next, by performing knot surgery on a homologically essential torus of self-intersection $0$, we obtain an infinite family of simply connected, symplectic and non-symplectic $4$-manifolds all homeomorphic but not diffeomorphic to\/ $X$.  Using Seiberg-Witten invariants, we distinguish their smooth structures.

\begin{thm}\label{theorem1} There exists an irreducible symplectic\/ $4$-manifold  $X$ homeomorphic but not diffeomorphic to\/ $3\CP\#15\CPb$ that can be obtained using the genus two Lefschetz fibration on $K3\#2\CPb$ over $S^{2}$ with global monodromy given by the relation $\varrho = (c_{5}c_{4}c_{3}c_{2}c_{1})^6 = 1$ in the mapping class group $M_2$ by applying six lantern substitutions. 
\end{thm}

In order to prove this theorem, we first prove the following two lemmas. 

\begin{lem}\label{threel} The word $c_{5}c_{4}c_{3}c_{2}c_{1}c_{5}c_{4}c_{3}c_{2}c_{1}$ in the mapping class group $M_2$ can be conjugated to contain the lantern relation in three different ways. \end{lem}

\begin{figure}[ht]
\begin{center}
\resizebox{!}{3.5cm}{\includegraphics{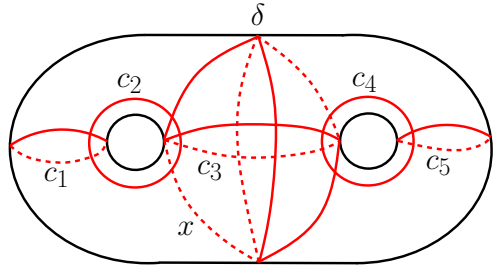}}
\caption{The Curves $c_1$, $c_2$, $c_3$, $c_4$, and $c_5$}
\label{fig:Sigma2}
\end{center}
\end{figure}

\begin{figure}[ht]
\begin{center}
\includegraphics[scale=.41]{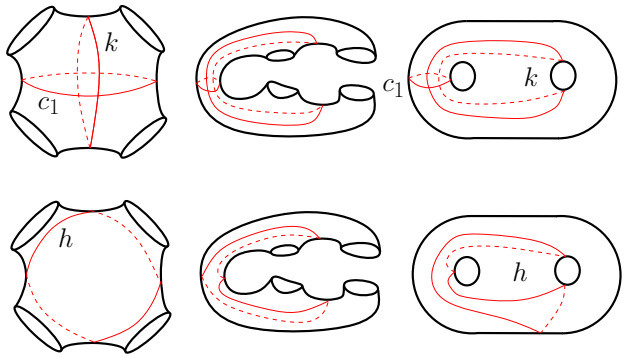}
\caption{Special Curves $k$, $h$}
\label{fig:Special1}
\end{center}
\end{figure}

\begin{figure}[ht]
\begin{center}
\includegraphics[scale=.41]{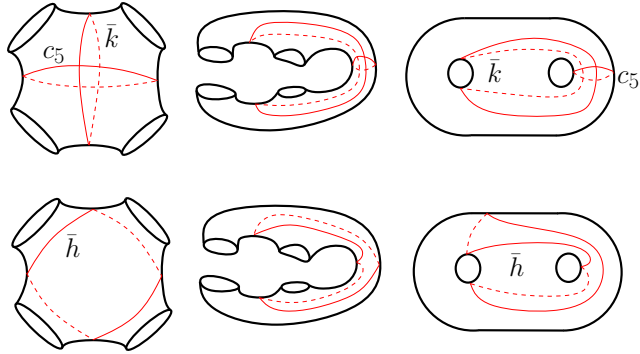}
\caption{Special Curves $\bar{k}$, $\bar{h}$}
\label{fig:Special2}
\end{center}
\end{figure}

\begin{proof}

Below we denote the lantern relation substitution, the braid relation substitution, the conjugation, and the arrangement using the commutativity by $ \overset{L} {\rightarrow} \ $ , $ \overset{B} {\rightarrow} \ $ , $ \overset{C} {\rightarrow} \ $,  $\sim$ respectively. For the convenience of the reader, we have highlighted the changes that occur in each step. \\

Let us consider the following cases\\

\begin{itemize}

\item Making the lantern substitution $c_{5} c_{1}^2 c_{5}$ for $\ c_{3} \delta x $.

 $ \textcolor{BrickRed}{c_{5}}c_{4}c_{3}c_{2}c_{1}c_{5}c_{4}c_{3}c_{2}\textcolor{BrickRed}{c_{1}}$ \\
$ \sim \ {}_{c_{5}}(c_{4}) \cdot c_{3}c_{2}c_{5}c_{1}\textcolor{BrickRed}{c_{5}c_{1}}c_{4}c_{3}\cdot{}_{c_{1}^{-1}}(c_{2})$ 
\\
$ \sim \ {}_{c_{5}}(c_{4}) \cdot c_{3}c_{2}\cdot \textcolor{BrickRed}{c_{5}^{2}c_{1}^{2}} \cdot c_{4}c_{3}\cdot{}_{c_{1}
^{-1}}(c_{2})$ \\
$ \overset{L} {\rightarrow} \ {}_{c_{5}}(c_{4}) \cdot c_{3}c_{2} \cdot c_{3} \delta x \cdot c_{4}c_{3}
\cdot{}_{c_{1}^{-1}}(c_{2})$ \\

\item Making the lantern substitution $c_{1} c_{3} c_{1} c_{3}$ for $ \bar{k} \bar{h} c_{5}$.

 $c_{5}c_{4}c_{3}c_{2}c_{1}\textcolor{BrickRed}{c_{5}c_{4}}c_{3}c_{2}c_{1}$ \\
$ \sim \ c_{5} c_{4} c_{5} \textcolor{BrickRed}{c_{3}} c_{2} c_{4} c_{1} c_{3} c_{2} \textcolor{BrickRed}{c_{1}}$ \\
$ \sim \ c_{5} c_{4} c_{5} \cdot {}_{c_{3}} (c_{2} c_{4}) \cdot \textcolor{BrickRed}{c_{1}^2 c_{3}^2}  \cdot {}_{c_{1}^{-1}}(c_{2}) $ \\
$ \overset{L} {\rightarrow} \ c_{5} c_{4} c_{5} \cdot {}_{c_{3}} (c_{2} c_{4}) \cdot \bar{k} \bar{h} c_{5} \cdot  {}_{c_{1}^{-1}}(c_{2}) $ \\

\item Making the lantern substitution $c_{3} c_{5}^2 c_{3}$ for $c_{1}kh$.

 $\textcolor{BrickRed}{c_{5}}c_{4}c_{3}c_{2}\textcolor{BrickRed}{c_{1}}c_{5}c_{4}\textcolor{BrickRed}{c_{3}}c_{2}c_{1}$ \\
$ \sim \ {}_{c_{5}} (c_{4}) \cdot \textcolor{BrickRed}{c_{5} c_{3}} c_{2} c_{5} c_{3} \cdot (c_{4})_{c_3} \cdot  c_{1} c_{2} c_{1}$ \\
$ \sim \ {}_{c_{5}} (c_{4}) \cdot {}_{c_{3}} (c_{2}) \cdot \textcolor{BrickRed}{c_{3}^2 c_{5}^2} \cdot {}_{c_{3}^{-1}}(c_{4}) \cdot  c_{1} c_{2} c_{1}$ \\
$ \overset{L} {\rightarrow} \ {}_{c_{5}} (c_{4}) \cdot {}_{c_{3}} (c_{2}) \cdot c_{1}  k h \cdot {}_{c_{3}^{-1}}(c_{4}) \cdot c_{1} c_{2} c_{1}$ \\
\end{itemize}

\end{proof}

Applying the above lemma several times, we next show how the word $\varrho = (c_{5}c_{4}c_{3}c_{2}c_{1})^6 = 1$ can be conjugated to contain six lantern relation.

\begin{lem}\label{sixl} The global monodromy of genus two Lefschetz fibration on $K3\#2\CPb$ over $S^{2}$ given by the relation $\varrho = (c_{5}c_{4}c_{3}c_{2}c_{1})^6 = 1$ can be conjugated and braid substituted to contain six lantern relations.
\end{lem}

\begin{proof}

We start with the identity word: $(c_{5}c_{4}c_{3}c_{2}c_{1})^6 = 1 $ \\

By applying three identities of Lemma~\ref{threel}, we can convert the word $(c_{5}c_{4}c_{3}c_{2}c_{1})^6 = ((c_{5}c_{4}c_{3}c_{2}c_{1})^2)^3$ into the following word:

\medskip

$ \sim \ {}_{c_{5}}(c_{4}) \cdot c_{3}c_{2} \cdot \textcolor{BrickRed}{c_{3}}  \delta x \cdot c_{4}c_{3}
\cdot {}_{c_{1}^{-1}}(c_{2}) \cdot c_{5} c_{4} \textcolor{BrickRed}{c_{5}} \cdot {}_{c_{3}} (c_{2} c_{4}) \cdot \bar{k} 
\bar{h} c_{5} \cdot {}_{c_{1}^{-1}}(c_{2}) \cdot {}_{c_{5}} (c_{4}) \cdot {}_{c_{3}} (c_{2}) \cdot c_
{1}  k h \cdot {}_{c_{3}^{-1}}(c_{4}) \cdot  c_{1} c_{2} c_{1}$ \\

$ \sim \ {}_{c_{5}}(c_{4}) \cdot c_{3}c_{2} \cdot {}_{c_{3}} ( \delta x ) \cdot {}_{c_{3}} (c_{4}) \cdot c_{3}^2
\cdot {}_{c_{1}^{-1}}(c_{2}) \cdot c_{5}^2 \cdot {}_{c_{5}^{-1}}(c_{4}) \cdot {}_{c_{3}} (c_{2} c_{4}) \cdot \bar{k} 
\bar{h} c_{5} \cdot {}_{c_{1}^{-1}}(c_{2}) \cdot {}_{c_{5}} (c_{4}) \cdot {}_{c_{3}} (c_{2}) \cdot c_
{1}  k h \cdot {}_{c_{3}^{-1}}(c_{4}) \cdot  \textcolor{BrickRed}{c_{1} c_{2} c_{1}}$ \\

$ \overset{B} {\rightarrow} \ {}_{c_{5}}(c_{4}) \cdot c_{3}c_{2} \cdot {}_{c_{3}} ( \delta x ) \cdot {}_{c_{3}} (c_{4}) \cdot c_{3}^2
\cdot {}_{c_{1}^{-1}}(c_{2}) \cdot \textcolor{BrickRed}{c_{5}^2 }\cdot {}_{c_{5}^{-1}}(c_{4}) \cdot {}_{c_{3}} (c_{2} c_{4}) \cdot \bar{k} 
\bar{h} c_{5} \cdot {}_{c_{1}^{-1}}(c_{2}) \cdot {}_{c_{5}} (c_{4}) \cdot {}_{c_{3}} (c_{2}) \cdot c_
{1}  k h \cdot {}_{c_{3}^{-1}}(c_{4}) \cdot  c_{2} c_{1} c_{2}$ \\

$ \sim \ {}_{c_{5}}(c_{4}) \cdot c_{3}c_{2} \cdot {}_{c_{3}} ( \delta x ) \cdot {}_{c_{3}} (c_{4}) \cdot \textcolor{BrickRed}{c_{3}^2c_{5}^2} \cdot {}_{c_{1}^{-1}}(c_{2}) \cdot {}_{c_{5}^{-1}}(c_{4}) \cdot {}_{c_{3}} (c_{2} c_{4}) \cdot \bar{k} 
\bar{h} c_{5} \cdot {}_{c_{1}^{-1}}(c_{2}) \cdot {}_{c_{5}} (c_{4}) \cdot {}_{c_{3}} (c_{2}) \cdot c_
{1}  k h \cdot {}_{c_{3}^{-1}}(c_{4}) \cdot  c_{2} c_{1} c_{2}$ \\

 $ \overset{L} {\rightarrow} \ {}_{c_{5}}(c_{4}) \cdot c_{3}c_{2} \cdot {}_{c_{3}} ( \delta x ) \cdot {}_{c_{3}} (c_{4}) \cdot \textcolor{BrickRed}{c_{1}}  k h 
\cdot {}_{c_{1}^{-1}}(c_{2}) \cdot {}_{c_{5}^{-1}}(c_{4}) \cdot {}_{c_{3}} (c_{2} c_{4}) \cdot \bar{k} 
\bar{h} c_{5} \cdot {}_{c_{1}^{-1}}(c_{2}) \cdot {}_{c_{5}} (c_{4}) \cdot {}_{c_{3}} (c_{2}) \cdot c_
{1}  k h \cdot {}_{c_{3}^{-1}}(c_{4}) \cdot  c_{2} c_{1} c_{2}$ \\

$ \sim \ {}_{c_{5}}(c_{4}) \cdot c_{3}c_{2} c_{1} \cdot {}_{c_{3}} ( \delta x ) \cdot {}_{c_{3}} (c_{4}) \cdot k h 
\cdot {}_{c_{1}^{-1}}(c_{2}) \cdot {}_{c_{5}^{-1}}(c_{4}) \cdot {}_{c_{3}} (c_{2} c_{4}) \cdot \bar{k} 
\bar{h} c_{5} \cdot {}_{c_{1}^{-1}}(c_{2}) \cdot {}_{c_{5}} (c_{4}) \cdot {}_{c_{3}} (c_{2}) \cdot c_
{1}  k h \cdot {}_{c_{3}^{-1}}(c_{4}) \cdot  c_{2} \textcolor{BrickRed}{c_{1} c_{2}}$ \\

$ \overset{C} {\rightarrow} \ \textcolor{BrickRed}{c_{1} c_{2}} \cdot {}_{c_{5}}(c_{4}) \cdot c_{3}c_{2} c_{1} \cdot {}_{c_{3}} ( \delta x ) \cdot {}_{c_{3}} (c_{4}) \cdot   k h 
\cdot {}_{c_{1}^{-1}}(c_{2}) \cdot {}_{c_{5}^{-1}}(c_{4}) \cdot {}_{c_{3}} (c_{2} c_{4}) \cdot \bar{k} 
\bar{h} c_{5} \cdot {}_{c_{1}^{-1}}(c_{2}) \cdot {}_{c_{5}} (c_{4}) \cdot {}_{c_{3}} (c_{2}) \cdot c_
{1}  k h \cdot {}_{c_{3}^{-1}}(c_{4}) \cdot  c_{2}$ \\

$ \sim \ {}_{c_{5}}(c_{4}) \cdot c_{1}\textcolor{BrickRed}{  c_{2} c_{3}c_{2}} c_{1} \cdot {}_{c_{3}} ( \delta x ) \cdot {}_{c_{3}} (c_{4}) \cdot   k h 
\cdot {}_{c_{1}^{-1}}(c_{2}) \cdot {}_{c_{5}^{-1}}(c_{4}) \cdot {}_{c_{3}} (c_{2} c_{4}) \cdot \bar{k} 
\bar{h} c_{5} \cdot {}_{c_{1}^{-1}}(c_{2}) \cdot {}_{c_{5}} (c_{4}) \cdot {}_{c_{3}} (c_{2}) \cdot c_
{1}  k h \cdot {}_{c_{3}^{-1}}(c_{4}) \cdot  c_{2}$ \\

$ \overset{B} {\rightarrow} \   {}_{c_{5}}(c_{4}) \cdot c_{1}  c_{3} c_{2}\textcolor{BrickRed}{c_{3} c_{1}} \cdot {}_{c_{3}} ( \delta x ) \cdot {}_{c_{3}} (c_{4}) \cdot   k h 
\cdot {}_{c_{1}^{-1}}(c_{2}) \cdot {}_{c_{5}^{-1}}(c_{4}) \cdot {}_{c_{3}} (c_{2} c_{4}) \cdot \bar{k} 
\bar{h} c_{5} \cdot {}_{c_{1}^{-1}}(c_{2}) \cdot {}_{c_{5}} (c_{4}) \cdot {}_{c_{3}} (c_{2}) \cdot c_
{1}  k h \cdot {}_{c_{3}^{-1}}(c_{4}) \cdot  c_{2}$ \\

$ \sim \  {}_{c_{5}}(c_{4}) \cdot \textcolor{BrickRed}{c_{1}^2  c_{3}^2} \cdot {}_{c_{1}^{-1}c_{3}^{-1}}(c_{2})  \cdot{}_{c_{3}} ( \delta x ) \cdot {}_{c_{3}} (c_{4}) \cdot   k h 
\cdot {}_{c_{1}^{-1}}(c_{2}) \cdot {}_{c_{5}^{-1}}(c_{4}) \cdot {}_{c_{3}} (c_{2} c_{4}) \cdot \bar{k} 
\bar{h} c_{5} \cdot {}_{c_{1}^{-1}}(c_{2}) \cdot {}_{c_{5}} (c_{4}) \cdot {}_{c_{3}} (c_{2}) \cdot c_
{1}  k h \cdot {}_{c_{3}^{-1}}(c_{4}) \cdot  c_{2}$ \\

$ \overset{L} {\rightarrow} \ {}_{c_{5}}(c_{4}) \cdot \bar{k} \bar{h} \textcolor{BrickRed}{c_{5}} \cdot {}_{c_{1}^{-1}c_{3}^{-1}}(c_{2})  \cdot{}_{c_{3}} ( \delta x ) \cdot {}_{c_{3}} (c_{4}) \cdot   k h 
\cdot {}_{c_{1}^{-1}}(c_{2}) \cdot {}_{c_{5}^{-1}}(c_{4}) \cdot {}_{c_{3}} (c_{2} c_{4}) \cdot \bar{k} \bar{h} c_{5} \cdot {}_{c_{1}^{-1}}(c_{2}) \cdot {}_{c_{5}} (c_{4}) \cdot {}_{c_{3}} (c_{2}) \cdot c_{1}  k h \cdot {}_{c_{3}^{-1}}(c_{4}) \cdot  c_{2}$ \\

$ \sim \ {}_{c_{5}}(c_{4}) \cdot \bar{k} \bar{h} \cdot {}_{c_{1}^{-1}c_{3}^{-1}}(c_{2})  \cdot{}_{c_{3}} ( \delta x ) \cdot {}_{c_{3} c_{5}} (c_{4}) \cdot   k h 
\cdot {}_{c_{1}^{-1}}(c_{2}) \cdot c_{4} \textcolor{BrickRed}{c_{5}} \cdot {}_{c_{3}} (c_{2} c_{4}) \cdot \bar{k} \bar{h} c_{5} \cdot {}_{c_{1}^{-1}}(c_{2}) \cdot {}_{c_{5}} (c_{4}) \cdot {}_{c_{3}} (c_{2}) \cdot c_{1}  k h \cdot {}_{c_{3}^{-1}}(c_{4}) \cdot  c_{2}$ \\

$ \sim \ {}_{c_{5}}(c_{4}) \cdot \bar{k} \bar{h} \cdot {}_{c_{1}^{-1}c_{3}^{-1}}(c_{2})  \cdot{}_{c_{3}} ( \delta x ) \cdot {}_{c_{3} c_{5}} (c_{4}) \cdot   k h 
\cdot {}_{c_{1}^{-1}}(c_{2}) \cdot\textcolor{BrickRed}{ c_{4}} \cdot {}_{c_{3}c_{5}} (c_{2} c_{4}) \cdot {}_{c_{5}} ( \bar{k} \bar{h} ) \cdot c_{5}^2 \cdot {}_{c_{1}^{-1}}(c_{2}) \cdot {}_{c_{5}} (c_{4}) \cdot {}_{c_{3}} (c_{2}) \cdot c_{1}  k h \cdot {}_{c_{3}^{-1}}(c_{4}) \cdot  c_{2}$ \\

$ \sim \ \textcolor{BrickRed}{c_{4}} \cdot {}_{c_{4}^{-1}} \textbf{(}{}_{c_{5}}(c_{4}) \cdot \bar{k} \bar{h} \cdot {}_{c_{1}^{-1}c_{3}^{-1}}(c_{2})  \cdot{}_{c_{3}} ( \delta x ) \cdot {}_{c_{3} c_{5}} (c_{4}) \cdot   k h 
\cdot {}_{c_{1}^{-1}}(c_{2})\textbf{)} \cdot  {}_{c_{3}c_{5}} (c_{2} c_{4}) \cdot {}_{c_{5}} ( \bar{k} \bar{h} ) \cdot c_{5}^2 \cdot {}_{c_{1}^{-1}}(c_{2}) \cdot {}_{c_{5}} (c_{4}) \cdot {}_{c_{3}} (c_{2}) \cdot c_{1}  k h \cdot {}_{c_{3}^{-1}}(c_{4}) \cdot c_{2} $ \\

$ \overset{C} {\rightarrow} \ {}_{c_{4}^{-1}} \textbf{(}{}_{c_{5}}(c_{4}) \cdot \bar{k} \bar{h} \cdot {}_{c_{1}^{-1}c_{3}^{-1}}(c_{2})  \cdot{}_{c_{3}} ( \delta x ) \cdot {}_{c_{3} c_{5}} (c_{4}) \cdot   k h 
 \cdot {}_{c_{1}^{-1}}(c_{2})\textbf{)} \cdot  {}_{c_{3}c_{5}} (c_{2} c_{4}) \cdot {}_{c_{5}} ( \bar{k} \bar{h} ) \cdot c_{5}^2 \cdot {}_{c_{1}^{-1}}(c_{2}) \cdot {}_{c_{5}} (c_{4}) \cdot {}_{c_{3}} (c_{2}) \cdot c_{1}  k h \cdot {}_{c_{3}^{-1}}(c_{4}) \cdot \textcolor{BrickRed}{c_{2}  c_{4}} $ \\

$ \sim \ {}_{c_{4}^{-1}} \textbf{(}{}_{c_{5}}(c_{4}) \cdot \bar{k} \bar{h} \cdot {}_{c_{1}^{-1}c_{3}^{-1}}(c_{2})  \cdot{}_{c_{3}} ( \delta x ) \cdot {}_{c_{3} c_{5}} (c_{4}) \cdot   k h 
\cdot {}_{c_{1}^{-1}}(c_{2})\textbf{)} \cdot  {}_{c_{3}c_{5}} (c_{2} c_{4}) \cdot {}_{c_{5}} ( \bar{k} \bar{h} ) \cdot c_{5}^2 \cdot {}_{c_{1}^{-1}}(c_{2}) \cdot {}_{c_{5}} (c_{4}) \cdot {}_{c_{3}} (c_{2}) \cdot c_{1}  k h \cdot {}_{c_{3}^{-1}}(c_{4}) \cdot c_{4}  \textcolor{BrickRed}{c_{2}} $ \\

$ \overset{C} {\rightarrow} \ \textcolor{BrickRed}{c_{2}} \cdot {}_{c_{4}^{-1}} \textbf{(}{}_{c_{5}}(c_{4}) \cdot \bar{k} \bar{h} \cdot {}_{c_{1}^{-1}c_{3}^{-1}}(c_{2})  \cdot{}_{c_{3}} ( \delta x ) \cdot {}_{c_{3} c_{5}} \cdot   k h 
\cdot {}_{c_{1}^{-1}}(c_{2})\textbf{)} \cdot  {}_{c_{3}c_{5}} (c_{2} c_{4}) \cdot {}_{c_{5}} ( \bar{k} \bar{h} ) \cdot c_{5}^2 \cdot {}_{c_{1}^{-1}}(c_{2}) \cdot {}_{c_{5}} (c_{4}) \cdot {}_{c_{3}} (c_{2}) \cdot c_{1}  k h \cdot {}_{c_{3}^{-1}}(c_{4}) \cdot  c_{4} $ \\

$ \sim \ {}_{c_{2}}\textbf{(} {}_{c_{4}^{-1}} \textbf{(}{}_{c_{5}}(c_{4}) \cdot \bar{k} \bar{h} \cdot {}_{c_{1}^{-1}c_{3}^{-1}}(c_{2})  \cdot{}_{c_{3}} ( \delta x ) \cdot {}_{c_{3} c_{5}} (c_{4}) \cdot   k h \cdot {}_{c_{1}^{-1}}(c_{2})\textbf{)} \cdot {}_{c_{3}c_{5}} (c_{2} c_{4}) \cdot {}_{c_{5}} ( \bar{k} \bar{h} )\textbf{)}
 \cdot c_{2} c_{5}^2 \cdot {}_{c_{1}^{-1}}(c_{2}) \cdot {}_{c_{5}} (c_{4}) \cdot {}_{c_{3}} (c_{2}) \cdot c_{1}  k h \cdot \textcolor{BrickRed}{{}_{c_{3}^{-1}}(c_{4}) \cdot c_{4}}$ \\

$ \overset{B} {\rightarrow} \ {}_{c_{2}}\textbf{(} {}_{c_{4}^{-1}} \textbf{(}{}_{c_{5}}(c_{4}) \cdot \bar{k} \bar{h} \cdot {}_{c_{1}^{-1}c_{3}^{-1}}(c_{2})  \cdot{}_{c_{3}} ( \delta x ) \cdot {}_{c_{3} c_{5}} (c_{4}) \cdot   k h \cdot {}_{c_{1}^{-1}}(c_{2})\textbf{)} \cdot {}_{c_{3}c_{5}} (c_{2} c_{4}) \cdot {}_{c_{5}} ( \bar{k} \bar{h} )\textbf{)}
 \cdot \textcolor{BrickRed}{c_{2}} c_{5}^2 \cdot {}_{c_{1}^{-1}}(c_{2}) \cdot {}_{c_{5}} (c_{4}) \cdot {}_{c_{3}} (c_{2}) \cdot c_{1}  k h \cdot c_{4} c_{3} $ \\

$ \sim \ {}_{c_{2}}\textbf{(} {}_{c_{4}^{-1}} \textbf{(}{}_{c_{5}}(c_{4}) \cdot \bar{k} \bar{h} \cdot {}_{c_{1}^{-1}c_{3}^{-1}}(c_{2})  \cdot{}_{c_{3}} ( \delta x ) \cdot {}_{c_{3} c_{5}} (c_{4}) \cdot   k h \cdot {}_{c_{1}^{-1}}(c_{2})\textbf{)} \cdot {}_{c_{3}c_{5}} (c_{2} c_{4}) \cdot {}_{c_{5}} ( \bar{k} \bar{h} )\textbf{)}
 \cdot c_{5}^2 \cdot {}_{c_{2} c_{1}^{-1}}(c_{2}) \cdot {}_{c_{5}} (c_{4}) \cdot \textcolor{BrickRed}{  c_{2} \cdot {}_{c_{3}} (c_{2})  }\cdot c_{1}  k h \cdot c_{4} c_{3} $ \\

$ \overset{B} {\rightarrow} \ {}_{c_{2}}\textbf{(} {}_{c_{4}^{-1}} \textbf{(}{}_{c_{5}}(c_{4}) \cdot \bar{k} \bar{h} \cdot {}_{c_{1}^{-1}c_{3}^{-1}}(c_{2})  \cdot{}_{c_{3}} ( \delta x ) \cdot {}_{c_{3} c_{5}} (c_{4}) \cdot   k h \cdot {}_{c_{1}^{-1}}(c_{2})\textbf{)} \cdot {}_{c_{3}c_{5}} (c_{2} c_{4}) \cdot {}_{c_{5}} ( \bar{k} \bar{h} )\textbf{)}
 \cdot \textcolor{BrickRed}{c_{5}^2} \cdot {}_{c_{2} c_{1}^{-1}}(c_{2}) \cdot {}_{c_{5}} (c_{4}) \cdot c_{3} c_{2} \cdot c_{1}  k h \cdot c_{4} c_{3} $ \\

$ \sim \ {}_{c_{2}}\textbf{(} {}_{c_{4}^{-1}} \textbf{(}{}_{c_{5}}(c_{4}) \cdot \bar{k} \bar{h} \cdot {}_{c_{1}^{-1}c_{3}^{-1}}(c_{2})  \cdot{}_{c_{3}} ( \delta x ) \cdot {}_{c_{3} c_{5}} (c_{4}) \cdot   k h \cdot {}_{c_{1}^{-1}}(c_{2})\textbf{)} \cdot {}_{c_{3}c_{5}} (c_{2} c_{4}) \cdot {}_{c_{5}} ( \bar{k} \bar{h} )\textbf{)}
  \cdot {}_{c_{2} c_{1}^{-1}}(c_{2}) \cdot {}_{c_{5}^{3}} (c_{4}) \cdot c_{5}^2 c_{3} c_{2} \cdot c_{1}  k h \cdot c_{4} \textcolor{BrickRed}{c_{3}} $ \\

$ \sim \ {}_{c_{2}}\textbf{(} {}_{c_{4}^{-1}} \textbf{(}{}_{c_{5}}(c_{4}) \cdot \bar{k} \bar{h} \cdot {}_{c_{1}^{-1}c_{3}^{-1}}(c_{2})  \cdot{}_{c_{3}} ( \delta x ) \cdot {}_{c_{3} c_{5}} (c_{4}) \cdot   k h \cdot {}_{c_{1}^{-1}}(c_{2})\textbf{)} \cdot {}_{c_{3}c_{5}} (c_{2} c_{4}) \cdot {}_{c_{5}} ( \bar{k} \bar{h} )\textbf{)}
  \cdot {}_{c_{2} c_{1}^{-1}}(c_{2}) \cdot {}_{c_{5}^{3}} (c_{4}) \cdot \textcolor{BrickRed}{c_{5}^2 c_{3}^2}  \cdot {}_{c_{3}^{-1}}(c_{2}) \cdot c_{1}  k h \cdot {}_{c_{3}^{-1}}(c_{4})  $ \\

$ \overset{L} {\rightarrow} \ \textcolor{BrickRed}{{}_{c_{2}}}\textcolor{BrickRed}{\textbf{(} {}_{c_{4}^{-1}} \textbf{(}}{}_{c_{5}}(c_{4}) \cdot \bar{k} \bar{h} \cdot {}_{c_{1}^{-1}c_{3}^{-1}}(c_{2})  \cdot{}_{c_{3}} ( \delta x ) \cdot {}_{c_{3} c_{5}} (c_{4}) \cdot   k h \cdot {}_{c_{1}^{-1}}(c_{2})\textcolor{BrickRed}{\textbf{)}} \cdot {}_{c_{3}c_{5}} (c_{2} c_{4}) \cdot {}_{c_{5}} ( \bar{k} \bar{h} )\textcolor{BrickRed}{\textbf{)}}
  \cdot {}_{c_{2} c_{1}^{-1}}(c_{2}) \cdot {}_{c_{5}^{3}} (c_{4}) \cdot c_{1} k h  \cdot {}_{c_{3}^{-1}}(c_{2}) \cdot c_{1}  k h \cdot {}_{c_{3}^{-1}}(c_{4})  $ \\

After suitable conjugations, we simplify the long conjugations highlighted above to acquire the following word with $24$ right handed Dehn twists. \\

 $ \sim \ {}_{c_{4}^{-1}c_{5}}(c_{4}) \cdot {}_{c_{2}c_{4}^{-1}}(\bar{k} \bar{h}) \cdot {}_{c_{2}c_{4}^{-1}c_{1}^{-1}c_{3}^{-1}}(c_{2})  \cdot {}_{c_{2}c_{4}^{-1}c_{3}}( \delta x ) \cdot {}_{c_{2}c_{4}^{-1}c_{3}c_{5}}(c_{4}) \cdot   {}_{c_{2}c_{4}^{-1}}(k h) \cdot {}_{c_{2}c_{1}^{-1}}(c_{2}) \cdot {}_{c_{2}c_{3}c_{5}}(c_{2} c_{4}) \cdot {}_{c_{2}c_{5}}( \bar{k} \bar{h} ) \cdot {}_{c_{2} c_{1}^{-1}}(c_{2}) \cdot {}_{c_{5}^{3}} (c_{4}) \cdot c_{1} k h  \cdot {}_{c_{3}^{-1}}(c_{2}) \cdot c_{1}  k h \cdot {}_{c_{3}^{-1}}(c_{4})  $ \\

Using a hyperelliptic signature formula for genus two Lefschetz fibration or by simple checking, we determine that $18$ of these Dehn twists are along nonseparating and $6$ are along separating curves. Notice that each lantern relation introduces one separating vanishing cycle. 

\end{proof}

Now we give a proof of main Theorem~\ref{theorem1}. 

\begin{proof} Let $X$ be the symplectic $4$-manifold obtained from $K3\#2\CPb$ by applying six lantern relation as in Lemma~\ref{sixl} above. Using Lemma~\ref{thm:rb}, we compute that
 
\begin{eqnarray*}
e(X) &=& e(K3\#2\CPb) - 6 = 26 - 6 = 20,\\
\sigma(X) &=& \sigma (K3\#2\CPb) + 6 = (-18) + 6 = -12.
\end{eqnarray*}

Freedman's theorem (cf.\ \cite{freedman}) implies that $X$ is homeomorphic to $3\CP\# 15\CPb$, once we show that $\pi_1(X)=1$. 

\begin{figure}[ht]
\begin{center}
\includegraphics[scale=.6]{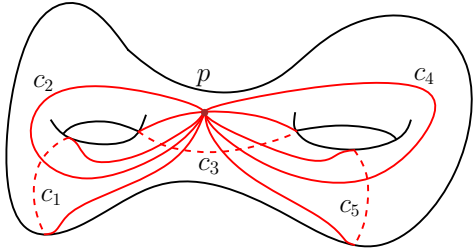}
\caption{The Loops $c_1$, $c_2$, $c_3$, $c_4$, and $c_5$}
\label{fig:basepoint}
\end{center}
\end{figure}

Let us denote a regular fiber of the genus two Lefschetz fibration on $X$ as $\Sigma_2$ and the generators of fundamental group of $\Sigma_2$ as ${c_{1}}$, ${c_{2}}$, ${c_{3}}$, ${c_{4}}$, and ${c_{5}}$. We continue to use the notation of the Figure~\ref{fig:Sigma2}, by choosing a base point $p$ as in $\Sigma_2$ as in Figure~\ref{fig:basepoint}. From the long exact homotopy sequence (see, for example\cite{GS}, page 290), we deduce that

\begin {center} 

$\pi_{1}(\Sigma_{2}) \longrightarrow \pi_{1}(X) \longrightarrow \pi_{1}(S^{2}) = 1$

\end{center}

\noindent Moreover,  $\pi_1(X)$ is finitely presented group and generated by the images of the standard generators, which we again denote by ${c_{1}}$, ${c_{2}}$, ${c_{3}}$, ${c_{4}}$, and ${c_{5}}$, under the surjection map. Furthermore, the loops $\beta_{1}$, $\beta_{2}$, $\beta_{3}$, $\cdots$, $\beta_{23}$, $\beta_{24}$ on the regular fiber are all nullhomotopic in $\pi_{1}(X)$, where $\beta_{1}$, $\beta_{2}$, $\beta_{3}$, $\cdots$, $\beta_{23}$, $\beta_{24}$ are the curves corresponding to $24$ Dehn twists $D_i$ of the genus two Lefschetz fibration on $X$ given below 

\medskip

\begin{gather}\label{eq: Dehn twists} \nonumber
D_1 = {}_{c_{4}^{-1}c_{5}}(c_{4}), \ \ D_2 = {}_{c_{2}c_{4}^{-1}} (\bar{k}), \ \ D_3 = {}_{c_{2}c_{4}^{-1}} (\bar{h}), \ \ D_4 = {}_{c_{2}c_{4}^{-1}c_{1}^{-1}c_{3}^{-1}} (c_{2}), \ \ D_5 = {}_{c_{2}c_{4}^{-1}c_{3}} (\delta), \\ \nonumber D_6 = {}_{c_{2}c_{4}^{-1}c_{3}} (x), \ \ D_7 = {}_{c_{2}c_{4}^{-1} c_{3}c_{5}} (c_{4}), \ \ D_8 = {}_{c_{2}c_{4}^{-1}}(k), \ \ D_9 = {}_{c_{2}c_{4}^{-1}} (h), \ \ D_{10} = {}_{c_{2} c_{1}^{-1}}(c_{2}), \\ \nonumber D_{11} = {}_{c_{2}c_{3}c_{5}} (c_{2}), \ \ D_{12} = {}_{c_{2}c_{3}c_{5}} (c_{4}), \ \ D_{13} = {}_{c_{2}c_{5}}(\bar{k}), \ \ D_{14} = {}_{c_{2}c_{5}} (\bar{h}), \ \ D_{15} = {}_{c_{2} c_{1}^{-1}}(c_{2}), \\ \nonumber D_{16} = {}_{c_{5}^{3}} (c_{4}), \ \ D_{17} = c_1, \ \ D_{18} = k, 
\ \ D_{19} = h, \ \ D_{20} = {}_{c_{3}^{-1}} (c_{2}), \\ \nonumber D_{21} = c_1, \ \ D_{22} = k, \ \ D_{23} = h, \ \ D_{24} = {}_{c_{3}^{-1}} (c_{4}). \nonumber
\end{gather}

\medskip

We observe that $\beta_{21} = 1$ in $\pi_{1}(X)$, implies that ${c_{1}}$ is trivial element in $\pi_{1}(X)$. Using $\beta_{10} = \beta_{20} = 1$, we obtain ${c_{2}}$ and ${c_{3}}$ are trivial in $\pi_{1}(X)$. Furthermore, the relations $\beta_{5} = \beta_{22} = \beta_{23} = \beta_{24} = 1$ imply that the elements ${c_{4}}$ and ${c_{5}}$ are null-homotopic and thus $X$ is simply connected.  

Using the blow up formula for the Seiberg-Witten function \cite{FS2}, we have $SW_{K3\#2\,\CPb}$ 
$= SW_{K3} \cdot \prod_{j=1}^{2}(e^{E_{i}} + e^{-E_{i}}) = (e^{E_{1}} + e^{-E_{1}})(e^{E_{2}} + e^{-E_{2}})$, where $E_{i}$ is an exceptional class coming from the $i^{th}$ blow up. Consequently, it follows from this formula, the set of basic classes of $K3\#2\,\CPb$ are given by $\pm E_{1} \pm E_{2}$, and the Seiberg-Witten invariants on these classes are $\pm 1$. Moreover, after performing two rational blowdowns along a copy of the configuration $C_2$, the resulting manifold is diffeomorphic to $K3$ by Lemma~\ref{n=2}. Thus, the only basic class is the zero class, which descends from the top classes ${\pm (E_{1} + E_{2})}$ in $K3\#2\,\CPb$. Next, using the Corollary 8.6 in~\cite{FS1}, we observe that $X$ has Seiberg-Witten simple type. Furthermore, by applying Theorem~\ref{SW1} and Theorem~\ref{SW2}, we completely determine the Seiberg-Witten invariants of $X$ using the basic classes and invariants of $K3$: Up to sign the symplectic manifold $X$ has only one basic class which descends from the canonical class of $K3$. By Theorem ~\ref{SW2} (or by Taubes theorem \cite{taubes}), the value of the Seiberg-Witten function on these classes, $\pm K_{X}$, are ${ \pm 1}$. 

Alternatively, we can determine the Seiberg-Witten invariants of $X$ directly by computing the algebraic intersection number of the classes $\pm E_{1} \pm E_{2}$, with the classes of $-4$ spheres of six $C_2$ configurations. Observe that these $-4$ spheres are the components of the singular fibers of $K3\#2\,\CPb$. Furthemore, by considering three regions on the genus two surface, where the rational blowdowns are performed (see Figures~\ref{fig:Special1} and \ref{fig:Special2}), and the location of the two points where we did blow up the genus two pencil (see Figures~\ref{fig:braided} and \ref{fig:genustwopencil}), we compute the intersection numbers as follows: Let $S$ denote the homology class of $-4$ sphere of $C_{2}$. We have $S \cdot E_{1} = S \cdot E_{2} = 1$. Consequently, $S \cdot \pm (E_{1} + E_{2}) = \pm 2$ and $S \cdot \pm (E_{1} - E_{2}) = 0$. Since among the four classes $\pm E_{1} \pm E_{2}$ only $E_{1} + E_{2}$ and $-(E_{1} + E_{2})$ have intersection $\pm 2$ with $-4$ spheres of $C_2$, it follows from Theorem~\ref{SW1} that these are only two classes that descend to $X$.          

Next, we apply the connected sum theorem for the Seiberg-Witten invariant and show that $SW$ function is trivial for $3\CP\# 15\CPb$. Since the Seiberg-Witten invariants are diffeomorphism invariants, we conclude that $X$ is not diffeomorphic to $3\CP\# 15\CPb$. 

The minimality of $X$ follows from the the fact that $X$ has no two basic classes $K$ and $K'$ such that $(K - K')^2 = -4$. Notice that $(K_{X} - (-K_{X}))^2 = 4({K_X}^{2}) = 16$ in our case.

\end{proof}

\begin{thm}\label{theorem2} There exist an infinite family of irreducible symplectic\/  and an infinite family of irreducible non-symplectic\/ pairwise non-diffeomorphic $4$-manifolds all homeomorphic to $X$.
\end{thm}

\begin{proof} We will use both the branched cover and the fiber sum descriptions of $K3\#2\CPb$\/ given as in Lemma~\ref{E} and Proposition~\ref{E1} to show that $X$\/ contains at least two disjoint tori that are disjoint from the singular fibers of genus two Lefschetz fibration on $K3\#2\CPb$\/ over $S^2$. An alternative proof, using the homology classes, is given in Example~\ref{Ex}. These tori descend from Gompf nucleus of $E(2) = K3$ (See Example~\ref{Ex}), and survive in $X$ after the rational blowdowns along $C_{2}$. Let us explain this more precisely using a geometric argument. First, note that according to the proof of Lemma~\ref{E} any genus two fiber, which arises from a two-fold cover of the sphere $h-e_{1}$, meets the torus $a \times b$ descending from an elliptic fiber of $K3$ at two points. Let $f_{1}, \cdots, f_{6}$ denote the complicated singular fibers of the genus two fibration on $K3\#2\CPb$\/. We perturb these singular fibers into Lefschetz type upon which we perform the rational blowdowns along $C_{2}$. Consider the tubular neighborhoods of these singular fibers in the manifold $K3\#2\CPb$\/. One should think of each neighborhood where we perturb one complicated singular fiber into five Lefschetz type singular fibers. The normal disks of these neighborhoods on the torus $a \times b$ are denoted in Figure~\ref{fig:rimtorus} as $\delta_{1}, \cdots, \delta_{6}$. The dots in the disks indicate there are $5$ singulars fibers over each disk $\delta_{i}$. We choose an open set $U$ on the torus $a \times b$ which contains the disks $\delta_{1}, \cdots, \delta_{6}$, and away from the loops $a$ and $b$ (See Figure~\ref{fig:rimtorus}). Thus, we can assume that the rational blowdown surgeries have no effect on the outside the tubular neighborhood of $U$. Next, we will choose a rim circle $\mu$ away from the tubular neighborhood of $U$. The rim tori that are not affected by six rational blowdowns are $\mu \times a$ and $\mu \times b$ (Figure~\ref{fig:rimtorus}). Note that each of these tori has a dual sphere of self-intersection $-2$, arising from the dual circles $b$ and $a$ (See Example~\ref{Ex}). These tori are Lagrangian, but we can perturb the symplectic form so that one of them, say $T = \mu \times a$ becomes symplectic. Moreover, $\pi_1(X \setminus T) = 1$, which follows from the Van Kampen's Theorem using the facts that $\pi_1(X) = 1$ and the rim torus has a dual sphere (see Proposition 1.2 in \cite{Ha}, or Gompf \cite{gompf}, page 564). Hence, we have a symplectic torus $T$ in $X$\/ of self-intersection $0$ such that $\pi_1(X \setminus T) = 1$. By performing a knot surgery on $T$, inside $X$, we acquire an irreducible $4$-manifold $X_K$ that is homeomorphic to $X$. By varying our choice of the knot $K$, we can realize infinitely many pairwise non-diffeomorphic $4$-manifolds, either symplectic or nonsymplectic.
\end{proof}

\begin{figure}[ht]
\begin{center}
\includegraphics[scale=.33]{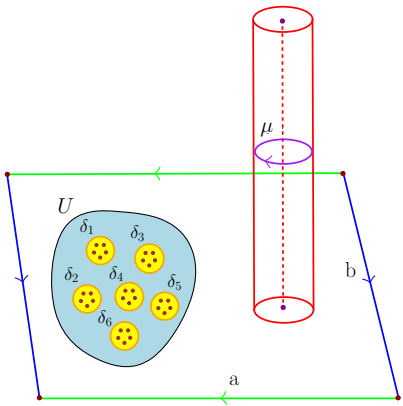}
\caption{The Rim Tori in $X$}
\label{fig:rimtorus}
\end{center}
\end{figure}

\begin{rmk} Let $X(m)$ denote the symplectic $4$-manifold obtained from $K3\#2\CPb$ by applying $m$ copies of the lantern substitutions, performed in the order given as in Lemma~\ref{sixl}, where $1 \leq m \leq 5$. Using Gompf's result~\ref{n=1} and~\ref{n=2}, we observe that $X(1)$ and $X(2)$ are diffeomorphic to $K3\#\CPb$ and $K3$ respectively. Similarly as in Theorem~\ref{theorem1}, we show that $X(m)$ is an exotic copy of $3\CP\# (21-m)\CPb$ for $3 \leq m \leq 5$. Moreover, the knot surgery on torus yields the infinite family of symplectic and non-symplectic irreducible $4$-manifolds all homeomorphic but not diffeomorphic to $3\CP\# (21-m)\CPb$. We leave the details to the reader as an exercise.

\end{rmk}

\section*{Acknowledgments} We would like to thank the referees for providing us with constructive and helpful comments and suggestions on the manuscript. A. Akhmedov was partially supported by Sloan Fellowship, NSF grants FRG-0244663 and DMS-1005741. J. Y. Park was partially supported by Undergraduate Research Opportunities Program Award at University of Minnesota.

\bibliographystyle{mrl}

\end{document}